\documentclass[12pt]{amsart}
\usepackage[utf8]{inputenc}
\usepackage[T1]{fontenc}
\usepackage{lmodern}
\usepackage{verbatim}
\usepackage{amssymb,amscd,amsmath,amsthm}
\usepackage{mathtools}
\usepackage[margin=1.2in]{geometry}
\usepackage{indentfirst}
\usepackage{enumerate}
\usepackage{cases}
\usepackage{xcolor}

\numberwithin{equation}{section}
\newtheorem{thm}{Theorem}[section]
\newtheorem{lem}{Lemma}[section]
\newtheorem{prop}[lem]{Proposition}

\newtheorem{conj}[lem]{Conjecture}

\newcommand{\eps}{\varepsilon}

\newcommand{\rmod}{\!\!\!\!\pmod}

\title[Moments of $k$-free numbers]{Moments of the distribution of $k$-free numbers in short intervals and arithmetic progressions}
\author{Ramon M. Nunes}
\address{Universidade Federal do Ceará// Departamento de Matemática// Av. Humberto Monte, s/n// Campus do Pici - Bloco 914// CEP: 60.440-900// Fortaleza - CE - Brasil }
\email{ramon@ufc.br}

\date{\today}

\begin{document}

\maketitle

\begin{abstract}
We show estimates for the distribution of $k$-free numbers in short intervals and arithmetic progressions. We argue that, at least in certain ranges, these estimates agree with a conjecture by H. L.  Montgomery.  
\end{abstract}

\section{Introduction}
Let $k\ge 2$. We say a positive integer $n$ is $k$-free if it is not divisible by any $k$-th power of an integer $>1$. The set of $k$-free numbers is known to have a natural density. Indeed, denoting by $\mu_k$ the characteristic function of $k$-free numbers, one has the asymptotic formula    
\begin{equation}\label{natural-density}
\sum_{n\leq X}\mu_k(n)=\zeta(k)^{-1}X+O(X^{1/k}),
\end{equation}
where $\zeta$ is the Riemann zeta function. Even though this is an elementary result, we mention that improving on the exponent $\frac{1}{k}$ is known to to have non-trivial implications towards the Riemann hypothesis. Conversely, under the Riemann hypothesis, Montgomery and Vaughan \cite{MV-kfree} have shown that the error term is $O(X^{\frac{1}{k+1}})$. This has been slightly improved over the years but not significantly so.

\subsection{Short Intervals}
If we switch to the question of $k$-free numbers lying in short intervals, much less is known. This was addressed by Hall in \cite{Hall1} and \cite{Hall2} only in the case $k=2$, but most of his results generalize straightforwardly to larger $k$. 

In what follows, we will define several quantities that depend upon an integer parameter $k\ge 2$ which, for convenience, we omit from the notation. We hope no confusion comes from this. We let for positive integers $k\ge 2$ and $n,H\ge 1$,   
$$
N(n,H):=\sum_{0\leq h<H }\mu_k(n+h),
$$
In view of \eqref{natural-density}, it is natural to expect that
$$
N(n,H)\sim \zeta(k)^{-1}H
$$
as $n$ and $H$ tend to infinity, with large uniformity. This leads to considering the discrepancy
$$
D(n,H):=N(n,H)-\zeta(k)^{-1}H.
$$

In trying to estimate the size of the set of $n$ for which $D(n,H)$ is large, Hall considered the moments
\begin{equation}\label{tildeM}
\mathfrak{M}(x,H,\ell):=\sum_{0<n\leq x} D(n,H)^{\ell},\;\;(\ell\geq 1).
\end{equation}
One obviously has the trivial bound
\begin{equation}\label{trivial-moment}
\mathfrak{M}(x,H,\ell)\ll_{k,\ell} xH^{\ell}
\end{equation}
where here and throughout the text, unless otherwise stated, the implied constant by the symbols $O$ and $\ll$ depend at most on the subscribed variables (e.g. the implied constant by the symbol $\ll_{\eps,k,\ell}$ depends at most on $\eps$, $k$ and $\ell$).

The bound \eqref{trivial-moment} should be far rom the truth as one expects that $D(n,H)$ tipically exhibits a lot of cancellation.
After some tedious but quite straightforward bookkeeping, the main results from \cite{Hall1} and \cite{Hall2}, when extrapolated to general $k$ show that there exist $\theta = \theta_{k,\ell}>0$ such that for $H\le x^{\theta}$, one has
$$
\mathfrak{M}(x,H,2)\sim CxH^{1/k}, 	
$$
for some $C=C_k>0$.
%
and, for each $\ell\geq 3$,
\begin{equation}\label{bound-Hall}
  \mathfrak{M}(x,H,\ell)\ll_{k,\ell} xH^{\left(\frac{\ell}{2}-\frac{k-1}{k}\right)}.
\end{equation}
When $k= 2$, a rather modest improvement was obtained in the author's thesis \cite{tese}. We now significantly improve on those bounds
\begin{thm}\label{thm-M3}
	Let $\mathfrak{M}(x,H,\ell)$ be defined as in \eqref{tildeM}. For every $\eps>0$, $k\ge 2$ and $\ell\geq 3$, we have
\begin{equation}\label{thm-M3-eq}
\mathfrak{M}(x,H,\ell)\ll _{\eps,k,\ell} x^{\eps} \left(H^{\frac{\ell}{2k}}x+H^{\ell}x^{\frac{2}{k+1}}\right),
\end{equation}
In particular, for $c_k=\frac{2k(k-1)}{(k+1)(2k-1)}$ and $1\leq H\leq x^{\frac{c_k}{\ell}}$, we have 
$$
\mathfrak{M}(x,H,\ell)\ll_{\eps,k,\ell} x^{1+\eps}H^{\frac{\ell}{2k}}.
$$
\end{thm}
\subsection{Arithmetic Progressions}

The distribution of arithmetic sequences in short intervals shares many similarities with the distribution of these sequences in arithmetic progressions. In the following we consider $k$-free numbers in arithmetic progressions, trying to emphasize as much as possible this analogy in the current case.

Let $X>1$ and let $q$ be a positive integer. It is easy to prove that the number of $k$-free numbers $\leq X$ which are relative prime to $q$ is asymptotically equivalent to $A_qX$, where
\begin{equation}
	A_q:=\frac{\varphi(q)}{q}\prod_{p\,\nmid\, q}\left( 1-\frac{1}{p^k} \right).
\end{equation}
The question of whether $k$-free numbers are well distributed among the arithmetic progressions modulo $q$ amounts to study the following error terms: Let $a$ be an integer such that $(a,q)=1$ and let $E(X,q,a)$ be given by the following equation
\begin{equation}\label{E}
	\sum_{\substack{n\leq X\\ n\equiv a\!\!\!\!\!\pmod q}}\mu_k(n)=\frac{A_q}{\varphi(q)}X + E(X,q,a).
\end{equation}
It is clear that $E(X,q,a)$ satisfies $E(X,q,a)\ll \frac Xq$ and we are interested in finding regions in the range $X$ and $q$ in which 
\begin{equation}\label{petit-o}
E(X,q,a)=o\left(X/q\right). 
\end{equation}
For $k=2$, it was shown in \cite{RMN23} that for every $\eps>0$, \eqref{petit-o} holds true as $X\rightarrow \infty$, for arbitrary $q\leq X^{\frac{25}{36}-\eps}$ and $(a,q)=1$, thus breaking the barrier of $X^{\frac23-\eps}$ that stood since the work of Prachar \cite{Prachar}. To the best of our knowledge, for $k\ge 3$, the best estimate to date is still the one obtained by completely elementary methods in \cite{Prachar} and which gives \eqref{petit-o} for $q\le X^{\frac{k}{k+1}-\eps}$. 
It is believed that \eqref{petit-o} holds as $X\rightarrow\infty$ for arbitrary $q\leq X^{1-\eps}$ and $(a,q)=1$. In fact, a probabilistic model suggests that something even stronger should hold: 

\begin{equation}\label{MontConj}
	E(X,q,a)=O_{\eps}\left(X^{\eps}(X/q)^{\frac{1}{2k}}\right)\text{, }\eps>0\text{ arbitrary}
\end{equation}
uniformly for $(a,q)=1$. When $k=2$, this conjecture was put forward by Le Boudec in \cite{LeBoudec} , and is just a slight modification of an older conjecture by Montgomery (see \cite[top of the page 145]{Croft}). We refer the interested reader to the discussion leading to \cite[Conjecture 1]{LeBoudec} for an explanation on why the previous conjecture might fail for large values of $q$.

We define the $\ell$-th moment  for this distribution as
\begin{equation}\label{M-Hq}
\mathcal{M}(X,q,\ell)=\sideset{}{^{\ast}}\sum_{a\!\!\!\!\pmod q}\!\!\!E(X,q,a)^{\ell},
\end{equation}
where the $*$ symbol means that we only sum over the classes that are relatively prime to $q$.

In view of \eqref{MontConj}, we suspect that the following should hold

\begin{conj}\label{Laconj}
	For integers $k\ge 2$ and $\ell\geq 1$, there exist $\delta = \delta_{k,\ell}>0$, $\eta = \eta_{k,\ell}>0$ and a positive multiplicative function $c = c_{k, \ell}$ such that
\begin{equation}
  \left|\mathcal{M}(X,q,\ell)-c(q) \varphi(q)\left(\frac Xq\right)^{\frac{\ell}{2k}}\right|\leq \varphi(q)\left(\frac Xq\right)^{\frac{\ell}{2k}-\eta},
\end{equation}
for every $q$ and $X$ satisfying 
$$
X^{1-\delta}\leq q\leq X^{1-\eps}\text{, }X\geq 2.
$$
Furthermore
\begin{itemize}
  \item if $\ell$ is even, $c(q)$ is bounded above and below by positive constants depending only on $k$ and $\ell$;
\item if $\ell$ is odd, $c\equiv 0$.
\end{itemize}
\end{conj}

When $k=\ell=2$, this was first shown in \cite{RMN} and the best admissible value of $\delta_{2,2}$ was subsequently improved in \cite{LeBoudec} and more recently in \cite{GMRR}. We also mention the results by Croft \cite{Croft} ($k=2$) and Vaughan \cite{Vaughan} where an extra sum is performed on the moduli $q$. It is quite remarkable that the result of \cite{GMRR} when summed over $q$ are still stronger than the one in \cite{Vaughan}.

Finally, very little was known for higher moments. As it transpires from our argument, the method of Hall could be easily imported to any $k\geq 2$. This would lead to the bounds
\begin{equation*}
	\mathcal{M}(X,q,\ell)\ll X^{\eps}\varphi(q)\left(\frac{X}{q}\right)^{\frac{\ell}{2}-\frac{k-1}{k}},
\end{equation*} 
for $q>X^{1-\delta}$ for certain $\delta=\delta_{k,\ell}$. As was the case for short intervals, slightly improved bounds were shown in \cite[Chapter 4]{tese}.

By an essentially identical proof to that of Theorem \ref{thm-M3}, we have
\begin{thm}\label{thm-M3-Hq}
Let $\mathcal{M}(X,q,\ell)$ be defined by \eqref{M-Hq}. Then for every $\eps>0$, every $k\ge 2$ and every $\ell\geq 3$, we have
\begin{equation}\label{thm-M3-eq-Hq}
  \mathcal{M}(X,q,\ell)\ll_{\eps,k,\ell} X^{\eps}\left( q\left(\frac{X}{q}\right)^{\ell/2k}+X^{\frac{2}{k+1}}\left(\frac{X}{q}\right)^{\ell-1}\right),
\end{equation}
In particular, for $c_k=\frac{2k(k-1)}{(k+1)(2k-1)}$, for every $\eps>0$, for every $\ell\geq 3$, for every $X\geq 1$ and every positive integer $q$ such that $X^{1-\frac{c_k}{\ell}}\leq q\leq X$, we have
$$
\mathcal{M}(X,q,\ell)\ll_{\eps,k,\ell}\varphi(q)\left(\frac{X}{q}\right)^{\frac{\ell}{2k}+\eps}.
$$
\end{thm}

Notice that according to Conjecture \ref{Laconj}, at least for even $\ell$, this estimate is sharp except for the term $X^{\eps}$. 


\subsection{Sums of singular series}

	Expanding  $D(n,H)^{\ell}$ and using asymptotic formulas for tuples of $k$-free numbers (\textit{e.g.} as in \cite{Mirsky2} or \cite{Tsang}), Hall deduced, for fixed integers $k,\ell \geq 1$, the asymptotic formula
\begin{equation}\label{Ml-Cl}
\mathfrak{M}(x,H,\ell)\sim C_{\ell}(H)x,\;\;\;\;\;\text{as }x\rightarrow \infty,
\end{equation}
where $C_{\ell}(H)$ is a combination of singular series terms (\textit{cf.} \cite[Eq. (1)]{Hall2}). A trivial bound in this case is
\begin{equation*}
C_{\ell}(H)\ll_{k,\ell} H^{\ell}
\end{equation*}
and in fact, most of the work in \cite{Hall2} is done in order to show estimates for $C_{\ell}(H)$. We take a similar route and thus, the source of our improvement comes from a better estimation of $C_{\ell}(H)$.  

The case of arithmetic progressions can be treated in a strikingly similar way. Expanding $E(X,q,a)^{\ell}$ and summing over $a\pmod{q}$, we are led to sums which are roughly like

$$
	\mathcal{S}_q(X,\mathbf{h}) \sum_{\substack{n\le X\\ (n,q)=1}}\mu_k(n+h_1q)\ldots \mu_k(n+h_jq), 
$$
for some tuple of integers $\mathbf{h}=(h_1,\ldots,h_j)$. We now need a mild generalization of the results from \cite{Mirsky2} that allows for the introduction of the condition $(n,q)=1$. This is furnished by Proposition \ref{alatsang} below. This shows that
$$
S_q(X,\mathbf{h}) \sim A_q(\mathbf{h})X,
$$
where $A_q(\mathbf{h})$ is the singular series term given by
\begin{equation}\label{Aq(h)}
	A_q\left(\mathbf{h}\right):=\frac{\varphi(q)}{q}\prod_{p\,\nmid\, q}\left(1-\frac{u_p(\mathbf{h})}{p^2}\right),
\end{equation}
and
\begin{equation}\label{def-u}
u_p(\mathbf{h})=\#\big\{h_1,\ldots,h_{\ell}\!\!\!\pmod{p^k}\big\},
\end{equation}

Applying this formula for several values of $\mathbf{h}$ along with the binomial theorem, we can show that the size of $\mathcal{M}(X,q,\ell)$ is controlled by $C_{\ell}(X/q,q)$, where for for positive integers $k\ge 2$, $\ell$ and $q$, and for $H\geq 1$, $C_{\ell}(H,q)$ is given by
\begin{equation}\label{def-frakC-Hq}
	C_{\ell}(H;q):=\sum_{j=0}^{\ell}\binom{\ell}{j}\left(-A_qH\right)^{\ell-j}B_{j}(H;q),
\end{equation}
with $B_{0}(H;q):=1$, and
\begin{equation}\label{Bj-Hq}
	B_{j}(H;q):=\int_{H-1}^{H}\sum_{0\le h_1< u}\ldots\sum_{0\le h_j< u}A_q(\mathbf{h})du,\text{ for }j\geq 1.
\end{equation}
The precise relationship between $\mathcal{M}(X,q,\ell)$ and $C_{\ell}(X,q)$ is given by formula \eqref{last-M3-Hq}. This is an effective analogue of \eqref{Ml-Cl} in the case of arithmetic progressions.\\
By letting $q=1$ and $H$ be an integer, this definition reduces to the one in \cite{Hall2} or, equivalently, the one in \ref{my-C}. In other words, we have
\begin{equation}\label{C,1=C}
C_{\ell}(H;1)=C_{\ell}(H)\text{, for every positive integer }k.
\end{equation}

It is thanks to identity \eqref{C,1=C} and the similarity of the structure of \eqref{Ml-Cl} and \eqref{last-M3-Hq} we are allowed to treat both short intervals and arithmetic progressions simultaneously. Indeed the following result will be a key ingredient in the proof of both Theorem \ref{thm-M3} and \ref{thm-M3-Hq}

\begin{thm}\label{frakC3-me-Hq}
For every $\eps>0$ and $k,\ell \ge 2$, we have the inequality
\begin{equation}\label{frakC3-eq-Hq}
  C_{\ell}(H;q)\ll_{\epsilon,k,\ell} H^{\frac{\ell}{2k}+\eps},
\end{equation}
\end{thm}
Unfortunately, it is not clear to us how to enhance this into an asymptotic formula for even $\ell$. Neither do we know how to get further cancellation when $\ell$ is odd.

\subsection{Other arithmetic functions}
There are plenty of results in the literature  concerning moments of other arithmetic functions. This means considering
$$
\mathfrak{M}_f(x,H,\ell) :=\sum_{0<n\leq x}\left( \sum_{h=0}^Hf(n+h)-\mathfrak{A}_f(x,H) \right)^{\ell},
$$
or
$$
\mathcal{M}_f(x,q,\ell) :=\sideset{}{^{\ast}}\sum_{a\rmod q}\left( \sum_{n\equiv a\rmod q}f(n)-\mathcal{A}_f(x,q) \right)^{\ell},
$$
where $\mathfrak{A}_f(x,H)$ and $\mathcal{A}_f(x,q)$ are the main terms which depend on the arithmetic function. In the following we discuss some of these result in some particularly interesting cases: the van Mangoldt function $\Lambda$, which is related to the distribution of primes, the M\"obius function $\mu$ and the divisor functions $d_k$.

\textbf{Primes:} Conditionally on an effective version of the Hardy-littlewood prime tuple conjecture, Montgomery and Soundararajan \cite{MontSound} have shown an asymptotic formula for the moments $\mathfrak{M}_{\Lambda}(x,H,\ell)$. The region of validity of the asymtptotic formula becomes narrower as $\ell$ grows but for $H$ tending to $\infty$ in a certain range, they deduce an approximately normal distribution for the error terms $ \sum_{h=0}^H\Lambda(n+h)-\mathfrak{A}_f(x,H) $ (\textit{cf.} \cite[Corollary 1]{MontSound}).

\textbf{Mobius:} Inspired by \cite{MontSound} and in view of the analogy between the study of the M\"obius and van Mangoldt functions, Ng \cite{Ng} has shown an asymptotic formula for the moments $ \mathfrak{M}_{\mu}(x,H,\ell) $. As a substitute for the prime tuple conjecture, Ng invokes an effective form of the Chowla conjecture on the combined oscillation of $\mu$ on tuples of integers. The calculation of main term is however much easier than the one in \cite{MontSound}. It only uses a formula for cunting tuples of squarefree numbers, such as Lemma \ref{alatsang}. 

\textbf{Divisor functions:} Finally, when $f=d_k$, there are unconditional results by Fouvry \textit{et al.} \cite{FGKM} and Kowalski-Ricotta \cite{KowalskiRicotta} concerning the moments $\mathcal{M}_{d_k}(x,q,\ell)$. The main tool responsible for evaluationg the main term and estimating the error is a bound for sums of products of trace functions which uses algebro-geometric techinques.

\subsection*{Plan of the paper}

After some preliminary lemmata in section \ref{lemmata}, the proofs of the main theorems are given in section \ref{proofs}.

\section{Preparatory results}\label{lemmata}

\subsection{Unfolding $A_q(\mathbf{h})$}\label{first}

It will not be very pratical for us to work with $A_q(\mathbf{h})$ under the form of an infinite product. By expanding the product and applying some finite Fourier analysis we may show that

\begin{lem}\label{Aqh}
Let  $\mathbf{h}=(h_1,\ldots,h_{\ell})\in\mathbb{Z}^{\ell}$, and let $A_q(\mathbf{h})$ be as in \eqref{Aq(h)}. Then we have the equality

\begin{equation*}
	A_q(\mathbf{h})=\frac{\varphi(q)}{q}\sum_{\substack{r_1=1\\(r_1,q)=1}}^{\infty}\ldots\sum_{\substack{r_{\ell}=1\\(r_{\ell},q)=1}}^{\infty}g_q(r_1)\ldots g_q(r_{\ell})\kappa(\mathbf{r};\mathbf{h}),
\end{equation*}
where $\mathbf{r}=(r_1,\ldots,r_{\ell})$, and

\begin{equation}\label{kappa}
\kappa(\mathbf{r};\mathbf{h})=
\sum_{\substack{\rho_i\in \mathcal{Q}(r_i^k)\\\rho_1+\ldots+\rho_{\ell}\in\mathbb{Z}}}e\left(\rho_1h_1+\ldots+\rho_{\ell}h_{\ell}\right) 
\end{equation} 
and finally,
\begin{equation*}
	\mathcal{Q}(r^2)=\left\{ \frac{a}{r^k};\;a\in \mathbb{Z},\;\mu_k((a,r^k))= 1 \right\} 
\end{equation*} 
and
\begin{equation*}
	g_q(r)=\frac{\mu(r)}{r^k}\prod_{p\nmid rq}\left( 1-\frac{1}{p^k} \right) 
\end{equation*}
\end{lem}

\begin{proof}
When $k=2$ and $q=1$, this is \cite[Lemma 1]{Hall2}, and the general case follows in a completely analogous way.
\end{proof}

Using Lemma \ref{Aqh} in definition \eqref{Bj-Hq}, we deduce that
\begin{equation*}
B_{\ell}(H;q)=\sum_{\substack{r_1=1\\(r_1,q)=1}}^{\infty}\ldots\sum_{\substack{r_{\ell}=1\\(r_{\ell},q)=1}}^{\infty}g_q(r_1)\ldots g_q(r_{\ell})Z_{\ell}(H;\mathbf{r}),
\end{equation*}
where $\mathbf{r}=(r_1,\ldots,r_{\ell})$ and

\begin{equation}\label{def-H}
	Z_{\ell}(H;\mathbf{r}):=\int_{H-1}^H\underset{\substack{\rho_i\in \mathcal{Q}(r_i^k)\\ \rho_1+\ldots +\rho_{\ell}\in \mathbb{Z}}}{\sum\ldots\sum}E_u(\rho_1)\ldots E_u(\rho_{\ell})\,du,
\end{equation}
with $E_u(\rho)$ given by

\begin{equation*}
E_u(\rho):=\sum_{h\le u}e(\rho h) 
\end{equation*}

\begin{lem}\label{C=sumE}
Let $C_{\ell}(H;q)$ be as in \eqref{def-frakC-Hq}. Then for every $H\geq 1$, and positive integers $k$, $\ell$ and $q$, the following holds
\begin{equation}\label{C=sumE-eq}
C_{\ell}(H;q)=\sum_{\substack{r_1\ge 2\\(r_1,q)=1}}\ldots\sum_{{\substack{r_{\ell}\ge 2\\(r_{\ell},q)=1}}}g_q(r_1)\ldots g_q(r_{\ell})Z_{\ell}(H;\mathbf{r}),
\end{equation}

\end{lem}

\begin{proof}
	As for the previous lemma, when $k=2$ and $q=1$ this follows from a result of Hall (see \cite[Lemma 2]{Hall2}) and the general case follows follows the exact same lines. 
\end{proof}

\subsection{Estimating $Z_{\ell}(H,\mathbf{r})$}
The core of the estimation for $C_{\ell}(H;q)$ comes from the Fundamental Lemma by Montgomery and Vaughan (\textit{cf.} \cite{MV}) which in the case $k=2$ was already a main ingredient in Hall's estimate in \cite{Hall2}.After some careful inspection of Hall's argument, one sees that the bound coming from the fundamental lemma is very wasteful when the variables $r_1,\ldots,r_{\ell}$ are small, specially for those terms for which we have many primes dividing many of the $r_i$. We complement those estimates with a completely elemetary result that is obtained without resorting to the fundamental lemma and only using positivity and classical bounds for linear exponential sums. We have

\begin{lem}\label{bound-for-Z}
	\begin{equation*}
		Z_{\ell}(H;\mathbf{r})\ll (r_1\ldots r_{\ell})^{k+\eps}\operatorname{min}\left(1,\frac{H^{\ell/2}}{[\mathbf{r}]^k} \right). 
	\end{equation*} 
\end{lem}

\begin{proof}
	The second estimate was proven by Hall in \cite{Hall2} (the details are carried out for $k=2$ but is is clear from the proof that the general case follows straightforwardly). In order to prove the first bound, we appeal to classical estimate
	\begin{equation*}
		|E_u(\rho)|\leq \|\rho\|^{-1}, 
	\end{equation*} 
	where $\|\cdot\|$ denotes the distance to the nearest integer. Applying this to \eqref{def-H} and using positivity, we have the inequality
\begin{align*}
	Z_{\ell}(H;\mathbf{r})&\leq \underset{\substack{\rho_i\in \mathcal{Q}(r_i^k)\\ \rho_1+\ldots +\rho_{\ell}\in \mathbb{Z}}}{\sum\ldots\sum}\|\rho_1\|^{-1}\ldots\|\rho_{\ell}\|^{-1}\\
												&\leq \prod_{i=1}^{\ell}\left(\sum_{a_i=1}^{r_i^k-1}\left\|\frac{a_i}{r_i^k}\right\|^{-1}\right).
\end{align*} 
It is a well-know fact that the inner sum above is 
\begin{equation*}
	\ll r_i^k\log(r_i^k)
\end{equation*} 
which is enough to prove the first estimate of the lemma, and this concludes the proof.
\end{proof}

\subsection{Tuples of squarefree numbers}

In this section we give an asymptotic formula for tuples of $k$-free numbers which is very similar to the main result of \cite{Mirsky2}. The sums considered here are slightly more general than those considered by Mirsky since we allow for an extra coprimality condition, but this geralizotion is rather straightforward. Much more important here is the fact that our bound, unlike the one in \cite{Mirsky2}, is uniform with respect to the "shifts".

	We point out that better bounds could probably be obtained by combining the ideas in \cite{Tsang} with a generalization of Heath-Brown's square sieve, such as in \cite{Brandes}, but since our focus is in the study of the sum of the singular series, we decided not to take that route. We did however benefit from some ideas taken from \cite{Tsang} in order to obtain estimates with an explicit dependency on the "shifts".

Our object of study in this section is the following counting function:

\begin{equation}\label{Sxqll}
	S_q(X_1,X_2,\mathbf{h}):=\sum_{\substack{X_1< n\le X_2\\(n,q)=1}}\mu_k(n+h_1q)\ldots\mu_k(n+h_jq) 
\end{equation}

Before giving an asymptotic formula for this object, we need a technical lemma.	This is a particular case of \cite[Lemma 3]{Mirsky2}, but we will rewrite the proof for convenience.

\begin{lem}\label{mirsky-lemma}
	Let $x>1$ and  $\mathbf{h}=(h_1,\ldots,h_j)$ be a $j$-tuple of distinct integers and let $X=\displaystyle\max_{1\le i\le j}(x+h_i)$. 
\begin{equation*}
\sum_{\substack{n\leq x\\d_i^k\mid n+h_i \\d_1d_2\ldots d_j>y\\ d_i \text{ pairwise coprime}}} 1 \ll_j x^{\eps}\left(Xy^{-k+1}+X^{\frac{2}{k+1}} \right) 
\end{equation*} 
\end{lem}

\begin{proof}

The proof is by induction. The case $j=1$ the left-hand side is bounded by

\begin{equation*}
	\sum_{y<d\le X^{1/k}} \left( \frac{x}{d^k} + O(1) \right)\ll xy^{-k+1}+X^{\frac{1}{k}},
\end{equation*} 
which is stronger than what is needed to show.

Now let $T(x,\mathbf{h},y)$ denote the sum on the left-hand side of the lemma. Suppose that we have the proved the assertion for $j\le j_0$ and let $z\ge X$ a parameter to be chosen later. Let $D=d_1\cdots d_{j_0+1}$ and let $T^1(x,\mathbf{h},y,z)$ the sum in $T(x,\mathbf{h},y)$ with the extra contidion that $D/d_i\le z$ for all $i$. It follows by an application of the Chinese remainder theorem that

$$
T^1(x,\mathbf{h},y,z)\le \sum_{ y<d_1\ldots d_{j_0+1}\le z}\left( \frac{x}{d_1^k\ldots d_{j_0+1}^k} + 1 \right)\ll X^{\epsilon} \left( xy^{-k+1} + z \right). 
$$

Now, by the hyperbola method, the remaining terms are bounded by
\begin{equation*}
	\sum_{\substack{n\leq x\\d_i^k\mid n+h_i' \\ d_1\ldots d_{j_0}> z^{\frac{j_0}{j_0+1}}\\ d_i \text{ pairwise coprime}}} 1\ll \	\sum_{\substack{n\leq x\\d_i^k\mid n+h'_i \\ d_1\ldots d_{j_0}> z^{\frac{j_0}{j_0+1}}\\ d_i \text{ pairwise coprime}}} \sum_{ d_{j_0+1}^k\mid n+h_{j_0+1}'}\ll X^{\eps}T(x,\mathbf{h}'',z^{\frac{j_0}{j_0+1}} ) 
\end{equation*}

for some permutation $\mathbf{h}'$ of $\mathbf{h}$ and $\mathbf{h}''=(h_1',..,h_{j_0}')$ and by the induction hypothesis, we see that the right-hand side of the above expression is

$$
\ll xw^{-k+1} + X^{\frac{2}{k+1}}
$$

Gathering these estimates, we see that for $j=j_0+1$,

\begin{equation*}
	T(x,\mathbf{h},y)	\ll X^{\eps}\left( xy^{-k+1} + X^{\frac{2}{k+1}} + xw^{-k+1} + w^{\frac{j_0+1}{j_0}} \right) 
\end{equation*}

Taking $w=X^{\frac{j_0}{kj_0+1}}$ gives the result.

\end{proof}
\begin{prop}\label{alatsang}
  Let $O<X_1<X_2$ and let $S_q(X_1,X_2,\mathbf{h}$ be as in \eqref{Sxqll}. Then we have the asymptotic formula 
\begin{equation*}
	S_q(X_1,X_2,\mathbf{h},q)=A_q(\mathbf{h})(X_2-X_1)+O\left( X^{\frac{2}{k+1}} \right),
\end{equation*} 
where $X=\displaystyle\max_{1\le i\le j}(X_2+h_iq)$.
\end{prop}

\begin{proof}
We start by defining, as in \cite{Reuss} and \cite{Tsang}, the following auxiliary functions:\\
Let
\begin{equation}\label{xi}
\sigma(n):=\prod_{p^k\mid n}p\:\text{ and }\: \xi(n)=\displaystyle\prod_{1\leq i\leq j}\sigma(n+h_iq).
\end{equation}

Notice that the function $\xi(n)$ above actually depends on $\mathbf{h}$ and $q$, but since these numbers will be held fixed in the following calculations, we omit this dependency.\\
Since
$$
\displaystyle\prod_{1\leq i\leq j}\mu_k(n+h_iq)=1\iff \xi(n)=1,
$$
we have
\begin{equation}\label{S=muNd}
S_q(X_1,X_2;\mathbf{h})=\displaystyle\sum_{\substack{X_1<n\leq X_2\\(n,q)=1}}\sum_{d\mid \xi(n)}\mu(d)=\sum_{\substack{d\geq 1\\ (d,q)=1}}\mu(d)N_d(X_1,X_2,q;\mathbf{h}),
\end{equation}
where
$$
N_d(X_1,X_2,q;\mathbf{h})=\{X_1<n\leq X_2; (n,q)=1\text{ and } \xi(n)\equiv 0\!\!\!\pmod{d}\}.
$$
Let $u_p = u_p(\mathbf{h})$ be as defined in \eqref{def-u}. Notice that, for $p$ coprime with $q$, the congruence
$$
\xi(n)\equiv 0\!\!\!\pmod{p}
$$
has exactly $u_p$ solutions for $n$ modulo $p^k$. Therefore, by the Chinese remainder theorem, the congruence
$$
\xi(n)\equiv 0\!\!\!\pmod{d}
$$
has
$$
U_d:=\prod_{p\mid d}u_p
$$
solutions $\!\!\!\pmod{d^2}$, for $d$ that are squarefree and satisfy $(d,q)=1$. As consequence,

\begin{equation}\label{Nd=}
N_d(X_1,X_2,q;\mathbf{h}):=\dfrac{\varphi(q)}{q}\dfrac{U_d}{d^k}(X_2-X_1)+O(\tau(q)U_d)
\end{equation}
uniformly for $0\leq X_1<X_2$ and squarefree $d$, satisfying $(d,q)=1$.

Let $y$ be a parameter to be chosen later depending on $X$. We write

\begin{equation}\label{S=S1+S2}
S(X_1,X_2,q;\mathbf{h})=S_1+S_2,
\end{equation}
where

\begin{equation*}
\begin{cases}
S_1=\displaystyle\sum_{\substack{1\leq d\leq y\\ (d,q)=1}}\mu(d)N_d(X_1,X_2,q;\mathbf{h}),\\
 \\
S_2=\displaystyle\sum_{\substack{y< d\leq X\\ (d,q)=1}}\mu(d)N_d(X_1,X_2,q;\mathbf{h}),
\end{cases}
\end{equation*}
and $X:=\displaystyle\max_{1\leq i\leq j}(X_2+h_iq)$ 

An application of \eqref{Nd=} gives

\begin{align}\label{S1=}
S_1&=\sum_{1\leq d\leq y}\mu(d)\left(\dfrac{\varphi(q)}{q}\dfrac{U_d}{d^k}(X_2-X_1)+O(\tau(q)U_d)\right)\notag\\
&=\dfrac{\varphi(q)}{q}(X_2-X_1)\sum_{\substack{d\geq 1\\ (d,q)=1}}\dfrac{\mu(d)U_d}{d^k}+O\left(X\sum_{d>y}\dfrac{U_d}{d^k}+\sum_{d\leq y}\tau(q)U_d\right)\notag\\
&=A_q(\mathbf{h})(X_2-X_1)+O_{\eps,j}(Xy^{-k+1+\eps}+\tau(q)y^{1+\eps}),
\end{align}
where in the last line we use that $U_d= \prod_{p\mid d}u_p\leq j^{\omega(d)}$, where $\omega(d)$ denotes the number of primes divisors of $d$.

Now since for large values of $d$ formula \eqref{Nd=} is not so meaningful, we must estimate $S_2$ differently. For each squarefree $d$ such that $d\mid \xi(n)$, we write

$$
d=\prod_{1\leq i\leq j}d_i,
$$
where $d_1,\ldots,d_j$ are such that

$$
d_i^k\mid n+h_iq,\;1\leq i\leq j.
$$
Remark that the decomposition above is in general not unique but we may bound the number of such decompositions by classical bounds or divisor functions. Thus, it follows from lemma \ref{mirsky-lemma} that we have the bound

\begin{equation}\label{S2=}
	S_2\ll X^{\eps}\left( \frac{X}{y} + X^{2/3} \right) 
\end{equation} 

Combining \eqref{S=S1+S2}, \eqref{S1=} and \eqref{S2=}, one deduces that

\begin{equation*}
	S(X_1,X_2,q,\mathbf{h})= A_q(\mathbf{h})\frac{\varphi(q)}{q}(X_2-X_1) + O_{\eps,j}\left(X^{\eps}\left( Xy^{-k+1} + y + X^{\frac{2}{k+1}} \right)   \right).
\end{equation*} 
The result now follow by taking, for instance, $y=X^{1/k}$.

\end{proof}

\section{Proofs of results}\label{proofs}
	
\subsection{Proof of Theorem \ref{frakC3-me-Hq}}\label{study}
By lemmas \ref{C=sumE} and \ref{bound-for-Z},
we have

$$
C_{\ell}(H;q)\ll \sum_{r_1\ge 2}\ldots\sum_{r_{\ell}\ge 2}(r_1..r_{\ell})^{\eps}\operatorname{min}\left(1,\frac{H^{\ell/2}}{[\mathbf{r}]^k} \right). 
$$

Now, by classical bounds for divisor functions, we may deduce that, if we break the sum according to the size of $\mathbf{r}$, we obtain the bound 

\begin{equation*}
C_{\ell}(H;q)\ll \sum_{r\le H^{\frac{\ell}{2k}}}r^{\eps} + \sum_{r> H^{\frac{\ell}{2k}}} \frac{H^{\ell/2}}{r^{k-\eps}} \ll H^{\frac{\ell}{2k}+\eps}.
\end{equation*}

\subsection{Proofs of Theorems \ref{thm-M3} and \ref{thm-M3-Hq}}\label{Theorems}

For a finite set $S$, $\# S$ denotes its cardinality and for an interval $I\subset \mathbb{R}$, $|I|$ denotes its length.

Let $0<H\leq x$, $H$ integer. We develop $\mathfrak{M}(x,H,\ell)$ (see \eqref{tildeM}) as follows
\begin{align}\label{develop-M3}
  \mathfrak{M}(x,H,\ell)&=\sum_{n\leq x}\left(\sum_{0\leq h\leq H}\mu_k(n+h)-\zeta(k)^{-1}H\right)^{\ell}\notag\\
												&=\sum_{j=0}^{\ell}\binom{\ell}{j}\left(-\zeta(k)^{-1}H\right)^{\ell-j}\underset{0\leq h_1,\ldots,h_j< H}{\sum\ldots\sum}S(0,x,1;\mathbf{h}).
\end{align}
where $\mathbf{h}=(h_1,\ldots,h_j)$ and $S_1(0,x;\mathbf{h})$ is as in \eqref{Sxqll}. It follows from Theorem \ref{alatsang} that the sum over the $h_1,\ldots,h_j$ equals

$$
B_j(H)x+O_{\eps,\ell}\left(H^jx^{\frac{2}{k+1}+\eps}\right),
$$
where
$$
B_j(H)=\underset{0\leq h_1,\ldots,h_j< H}{\sum\ldots\sum}A(\mathbf{h}). 
$$
By formula \eqref{develop-M3}, we obtain the equality

\begin{equation}\label{last-M3}
\mathfrak{M}(x,H,\ell)=C_{\ell}(H)x+O_{\eps,\ell}\left(k^{\ell}x^{\frac23+\eps}\right),
\end{equation}
where 
\begin{equation}\label{my-C}
C_{\ell}(H)=\sum_{ j=0}^{\ell}\binom{\ell}{j}\left( -\zeta(k)^{-1}H \right)^{\ell-j}  B_j(H).
\end{equation} 
Theorem \ref{thm-M3} is now a simple consequence of Theorem \ref{frakC3-me-Hq}, \eqref{C,1=C} and equation \eqref{last-M3} above.\\

Theorem \ref{thm-M3-Hq} follows in a much similar fashion. We start, as before, by expanding $\mathcal{M}(X,q,\ell)$ (see \eqref{M-Hq}). We obtain the formula

\begin{align}\label{develop-M3-Hq}
	\mathcal{M}(X,q,\ell)&=\sum_{j=0}^{\ell}\binom{\ell}{j}\left(-A_q\frac{X}{\varphi(q)}\right)^{\ell-j}\underset{\substack{0< n_1,\ldots,n_j\leq X\\n_1\equiv\ldots\equiv n_j\!\!\!\!\pmod{q} \notag \\ (n_j,q)=1}}{\sum\ldots\sum}\mu_k(n_1)\ldots\mu_k(n_j)\\
												 &=:\sum_{j=0}^{\ell}\binom{\ell}{j}\left(-A_q\frac{X}{\varphi(q)}\right)^{\ell-j}\mathcal{S}_j(X;q),
\end{align}
say.

We make the change of variables $n_j=n$ and $n_i=n_j+f_iq$, for $i=1,\ldots,j-1$. Thus, we are allowed to write
\begin{equation}\label{firstI}
  \mathcal{S}_j(X;q)=\underset{-\frac{X}{q}\leq f_1,\ldots,f_{j-1}\leq \frac{X}{q}}{\sum\ldots\sum}\sum_{\substack{\;n\in I(X,q;(\mathbf{f},0))\\(n,q)=1}}\mu_k(n+f_1q)\ldots\mu_k(n+f_{j-1}q)\mu_k(n),
\end{equation}
where for every $j-$tuple of integers $\mathbf{h}=(h_1,\ldots,h_j)$, we write
\begin{equation*}
I(X,q;\mathbf{h}):=\displaystyle\bigcap_{i=1}^{j}(-h_iq,X-h_iq].
\end{equation*}
Note that whenever $I(X,q;\mathbf{h})\neq \emptyset$, we have $I(X,q;\mathbf{h})=(X_1,X_2]$, where $X_1$, $X_2$ are real numbers satisfying
$$
0\leq X_1+h_iq<X_2+h_iq\leq X,\;i=1,\ldots,j.
$$
Hence, we may use Theorem \ref{alatsang} for the inner sum on the right-hand side of \eqref{firstI}. After summing over $f_1,\ldots,f_{j-1}$, we see that
\begin{align}\label{b4lastlem}
\mathcal{S}_j(X;q)&=\underset{-\frac{X}{q}\leq f_1,\ldots,f_{j-1}\leq \frac{X}{q}}{\sum\ldots\sum}\Bigg(A_q((\mathbf{f},0))\left|I(X,q;(\mathbf{f},0))\right|+O_{\eps,\ell}\left(X^{\frac23+\eps}\right)\Bigg)\notag\\
&=\sum_{\mathbf{f}\in\mathbb{Z}^{j-1}}A_q((\mathbf{f},0))\left|I(X,q;(\mathbf{f},0))\right|+O_{\eps,\ell}\left(X^{\frac23+\eps}\left(\frac{X}{q}\right)^{j-1}\right),
\end{align}
where, in the second line, we observed that whenever $|f_i|>\frac{X}{q}$ for some $1\leq i\leq j-1$, then $\left|I(X,q;(\mathbf{f},0))\right|=0$.

In what follows next, we evaluate the sum over the $f_i$ above. To this purpose we have the following:
\begin{lem}\label{lastlemma}
With the above notation, and $B_j$ as defined in \eqref{Bj-Hq}, we have for every $X>0$ and every integer $q$, the equality
 
$$
\sum_{\mathbf{f}\in\mathbb{Z}^{\ell-1}}A_q((\mathbf{f},0))\left|I(X,q;(\mathbf{f},0))\right|=qB_j\left(\frac{X}{q};q\right)
$$
\end{lem}

\begin{proof}
 
We first notice that

\begin{align}
\left|I(X,q;(\mathbf{f},0))\right|&=\int_{-\infty}^{+\infty}\chi_{(0,X]}(v)\prod_{i=1}^{j-1}\chi_{(-f_iq,X-f_iq]}(v) dv\notag\\
&= q\sum_{f=-\infty}^{\infty}\int_0^1\chi_{(0,X]}(qu+qf)\prod_{i=1}^{j-1}\chi_{\left(-f_iq,X-f_iq\right]}(qu+qf)\,du\notag\\
&= q\sum_{f=-\infty}^{\infty}\int_0^1\chi_{(-f,\frac{X}{q}-f]}(u)\prod_{i=1}^{j-1}\chi_{\left(-f-f_i,\frac{X}{q}-f-f_i\right]}(u)\,du,
\end{align}
where for each measurable set $A\subset \mathbb{R}$, $\chi_A$ denotes its characteristic function and in the second line we made the change of variables $v=qu+qf$, where $u\in (0,1]$ and $f\in \mathbb{Z}$. Summing over $f_1,\ldots,f_{j-1}$ and making the change of variables

$$
\begin{cases}
 h_i=f+f_i,1\leq i\leq j-1,\\
 h_j=f,
\end{cases}
$$
we obtain, since $A_q((\mathbf{f},0))=A_q(\mathbf{h})$ (recall \eqref{Aq(h)}),

\begin{align*}
\sum_{\mathbf{f}\in\mathbb{Z}^{j-1}}A_q((\mathbf{f},0))\left|I(X,q;(\mathbf{f},0))\right|&=q\sum_{\mathbf{h}\in\mathbb{Z}^{j}}A_q(\mathbf{h})\int_0^1 \prod_{i=1}^j\chi_{(-h_i,\frac{X}{q}-h_i]}(u)\,du\\
&=q\sum_{\mathbf{h}\in\mathbb{Z}^{j}}A_q(\mathbf{h})\int_0^1 \prod_{i=1}^j\chi_{(-u,\frac{X}{q}-u]}(h_i)\,du.
\end{align*}
Finally, by interchanging the order of summation and integration, we see that

$$
\sum_{\mathbf{f}\in\mathbb{Z}^{j-1}}A_q((\mathbf{f},0))\left|I(X,q;(\mathbf{f},0))\right|=q\int_0^1 \underset{-u<h_1,\ldots,h_j\leq H-u}{\sum\ldots \sum}A_q(\mathbf{h})du,
$$
which concludes the proof of Lemma \ref{lastlemma}.

\end{proof}

Now, Lemma \ref{lastlemma} when applied in \eqref{b4lastlem} gives the equality

\begin{equation}\label{FIN}
 \mathcal{S}_j(X;q)=qB_j\left(\frac{X}{q};q\right)+O_{\eps,\ell}\left(X^{\frac{2}{k+1}+\eps}\left(\frac{X}{q}\right)^{j-1}\right).
\end{equation}
Thus, by \eqref{FIN} in \eqref{develop-M3-Hq} one obtains (recall \eqref{def-frakC-Hq})

\begin{equation}\label{last-M3-Hq}
\mathcal{M}_{\ell}(X;q)=qC_{\ell}\left(\frac{X}{q};q\right)+O_{\eps,\ell}\left(X^{\frac{2}{k+1}+\eps}\left(\frac{X}{q}\right)^{\ell-1}\right).
\end{equation}
Theorem \ref{thm-M3-Hq} now follows from Theorem \ref{frakC3-me-Hq} and equation \eqref{last-M3-Hq} above.\\

\section*{Acknowledgements}

A first version of this result is part of the author's thesis. It is a pleasure to thank \'Etienne Fouvry for his guidance during that period.

\end{document}